\documentclass[12pt]{amsart}
\usepackage{amsmath}
\usepackage{verbatim}
\usepackage{amsfonts}
\usepackage{amssymb}
\usepackage[dvipsnames]{xcolor}
\usepackage[all]{xy}
\usepackage{enumerate}
\usepackage[top=1in, bottom=1in, left=0.9in, right=0.9in]{geometry}
\usepackage{mathrsfs}

\usepackage{stmaryrd}
\usepackage{bbold}

\usepackage{soul}
\usepackage[colorlinks=true,citecolor=PineGreen,linkcolor=RawSienna]{hyperref}
\usepackage{mabliautoref}
\input{kmacros3.sty}

\numberwithin{equation}{section}

\usepackage{mabliautoref}

\renewcommand{\m}{\mathfrak{m}}

\newcommand{\pd}{\operatorname{pd}}
\newcommand{\fd}{\operatorname{fd}}

\renewcommand{\Spec}{\operatorname{Spec}}

\newcommand{\depth}{\operatorname{depth}}

\renewcommand{\leq}{\leqslant}
\renewcommand{\geq}{\geqslant}

\title{A Kunz-type characterization of regular rings via alterations}
\author{Linquan Ma}
\thanks{Ma was supported in part by NSF Grant DMS \#1836867/1600198.}
\author{Karl Schwede}
\thanks{Schwede was supported in party by NSF CAREER Grant DMS \#1252860/1501102 and NSF Grant \#1801849.}
\address{Department of Mathematics\\ Purdue University\\  West Lafayette\\ IN 47907}
\email{ma326@purdue.edu}
\address{Department of Mathematics\\ University of Utah\\ Salt Lake City\\ UT 84112}
\email{schwede@math.utah.edu}

\subjclass[2010]{14F18, 13A35, 14F17, 13D05, 13D45, 14B15, 14B05}
\keywords{multiplier ideals, projective dimension, regular rings, rational singularities}

\begin{document}

\begin{abstract}
We prove that a local domain $R$, essentially of finite type over a field, is regular if and only if for every regular alteration $\pi : X \to \Spec R$, we have that $\myR \pi_* \O_X$ has finite (equivalently zero in characteristic zero) projective dimension.
\end{abstract}

\maketitle

\section{Introduction}

In \cite{KunzCharacterizationsOfRegularLocalRings} Kunz proved that a Noetherian ring of characteristic $p > 0$ is regular if and only if the $e$-th iterated Frobenius map
\[
\xymatrix@R=3pt{
R \ar[r] & R \\
x \ar@{|->}[r] & x^{p^e}
}
\]
is flat for some, or equivalently every, $e > 0$. This is generalized in \cite{RodicioOnaResultAvramov}: the condition $\fd_RR^{(e)}<\infty$ (for some, or equivalently every, $e>0$) implies $R$ is regular, where $R^{(e)}$ denotes the target of the $e$-th Frobenius map. Moreover, the direct limit of $R^{(e)}$ is the perfection $R^\infty$ of $R$. Kunz's theorem can also be generalized using perfection: $R$ is regular if and only if $\fd_RR^{\infty}<\infty$ \cite{AberbachLiAsymptoticVanishingRegular,BhattIyengarMaRegularRingsPerfectoidAlgebras}.

However, in all these characterizations, the Frobenius map plays a prominent role and hence they do not extend to characteristic zero. In this paper, motivated by the connections between multiplier ideals and test ideals \cite{BlickleSchwedeTuckerTestAlterations}, we prove the following characterization of regularity using alterations.

\begin{mainthm*}
Suppose $(R, \fram)$ is a local domain essentially of finite type over a field.  Then the following conditions are equivalent.
\begin{itemize}
\item[(a)]  $R$ is regular.
\item[(b)]  For every regular alteration $\pi : X \to \Spec R$, $\pd_R \myR \pi_* \O_X < \infty$ (i.e., the derived image of the structure sheaf has finite projective dimension).
\end{itemize}
Moreover, when $R$ has characteristic $0$, the above are also equivalent to
\begin{itemize}
\item[(c)]  For every regular alteration $\pi : X \to \Spec R$, $\pd_R \myR \pi_* \O_X = 0$.
\end{itemize}
\end{mainthm*}

This result is also motivated by the fact that there is close connection between big Cohen-Macaulay algebras and $\myR \pi_* \O_X$ (recall that the latter is a Cohen-Macaulay complex, \cite{RobertsCohenMacaulayComplexes}).

In fact, this theorem in characteristic $p>0$ essentially follows from a result of Bhatt on killing cohomology using finite covers \cite{BhattDerivedsplintersinpositivecharacteristic} and the characterization of regularity using $R^+$, see \cite{AberbachLiAsymptoticVanishingRegular,BhattIyengarMaRegularRingsPerfectoidAlgebras}. Our main contribution is the characteristic zero case.  Meanwhile, it is quite natural to ask whether the same characterization of regularity holds in mixed characteristic:

\begin{question}
Suppose $(R, \fram)$ is an excellent local domain essentially of finite type over the integers. If for every regular alteration $\pi : X \to \Spec R$, we have that $\pd_R \myR \pi_* \O_X < \infty$, then is $R$ regular?
\end{question}

See \autoref{rmk.MixedCharacteristicSurface} for some additional discussion.  Finally, one should also compare with the classical result that if $R \subseteq S$ is a flat local extension and $S$ is regular, then so is $R$, see \cite[Theorem 23.7]{MatsumuraCommutativeRingTheory}.

\subsubsection*{Acknowledgements}  We thank Bhargav Bhatt and Srikanth Iyengar for comments on a previous draft of this article.

\section{Preliminaries}
Suppose $R$ is a Noetherian ring. For a bounded above cochain complex of $R$-modules $C$ ($\dots \to C^{-1} \to C^0 \to C^1 \to \dots$), we define the \emph{projective dimension of $C$}, denoted $\pd_R C$ to be
\[
\inf \{ \sup \{ i \;|\; P^{-i} \neq 0 \} \;|\; P^{\bullet} \text{ is a projective resolution of $C$} \}.
\]
Here a projective resolution is a cochain complex of projective modules quasi-isomorphic to $C$. If $(R, \fram)$ is additionally local, we also define
\[
\depth C = \inf \{ i \;|\; H^i_{\fram}(C) \neq 0 \}.
\]
These are natural extensions of the projective dimension and the depth of a finitely generated $R$-module.

We will need the following result of Foxby and Iyengar, which is a vast generalization of the classical Auslander-Buchsbaum formula.
\begin{theorem}[\cite{FoxbyIyengarDepthAmplitude}]
\label{thm.BuchEisComplex}
Let $(R,\m)$ be a local ring and let $M$ and $P$ be complexes of $R$-modules.  If $\pd_R P$ is finite and ${\mathrm{H}}(P)$ is nonzero and finitely generated, then
\[
\depth_R M = \depth_R(M \otimes^{\bf L}_R P) + \pd_R P.
\]
\end{theorem}

We can now prove our main result in positive characteristic.

\begin{proof}[Proof of Main Theorem in characteristic $p>0$]
First of all by \autoref{thm.BuchEisComplex}, $$\pd_R \myR\pi_*\O_X+\depth \myR\pi_*\O_X=\depth R.$$
Since $\depth \myR\pi_*\O_X\geq 0$ (as $\myR\pi_*\O_X$ lives in positive degree), we know that for all regular alteration, $\pd_R\myR\pi_*\O_X\leq \depth R$.

By \cite[Theorem 1.5]{BhattDerivedsplintersinpositivecharacteristic}, for every regular alteration $\pi$: $X\to \Spec R$, there exists another regular alteration $\pi'$: $Y\xrightarrow{f} X\to \Spec R$ such that the map $\tau_{\geq 1}\myR \pi_*\O_X\to \tau_{\geq 1}\myR \pi'_*\O_Y$, induced by the diagram of triangles below, is $0$.
\[\xymatrix{
\pi_*\O_X \ar[r] \ar[d] & \myR \pi_*\O_X \ar[r] \ar[d] & \tau_{\geq 1}\myR \pi_*\O_X \ar[d]^0 \ar[r]^{+1} & \\
\pi'_*\O_Y \ar[r] & \myR \pi'_*\O_Y \ar[r] & \tau_{\geq 1}\myR \pi'_*\O_Y \ar[r]_{+1} &
}
\]
Tensoring with $k=R/\m$ and taking cohomology, for all $i>\depth R$ we get:
\footnotesize
\[\xymatrix{
0=\myH^{-i-1}(\myR \pi_*\O_X\otimes^\myL k) \ar[r]& \myH^{-i-1} (\tau_{\geq 1}\myR \pi_*\O_X \otimes^\myL k)\ar[d]^0 \ar[r] & \myH^{-i}(\pi_*\O_X\otimes^\myL k) \ar[r] \ar[d] & 0=\myH^{-i}(\myR \pi_*\O_X\otimes^\myL k)  \\
0=\myH^{-i-1}(\myR \pi'_*\O_Y\otimes^\myL k) \ar[r]& \myH^{-i-1} (\tau_{\geq 1}\myR \pi_*\O_X \otimes^\myL k)   \ar[r] & \myH^{-i}(\pi'_*\O_Y\otimes^\myL k) \ar[r] & 0=\myH^{-i}(\myR \pi_*\O_X\otimes^\myL k)
}
\]
\normalsize

An easy diagram chasing shows that for all regular alteration $X$, we can find another regular alteration $Y$ such that $\myH^{-i}(\pi_*\O_X\otimes^\myL k)\to  \myH^{-i}(\pi'_*\O_Y\otimes^\myL k)$ is $0$.  By writing $R^+$ as the colimit over finite domain extensions of $R$, $R^+ = \lim_{S \supseteq R} S$, we see that $\myH^{-i}(R^+\otimes^\myL k)=\Tor_i^R(R^+,k)=0$ for all $i>\depth R$. Now by \cite[Corollary 3.5]{AberbachLiAsymptoticVanishingRegular} or \cite[Theorem 4.13]{BhattIyengarMaRegularRingsPerfectoidAlgebras}, $R$ is regular.
\end{proof}

We recall the following very useful result of Corso-Huneke-Katz-Vasconcelos.
\begin{theorem}\textnormal{(\cite[Corollary 3.3]{CorsoHunekeKatzVasconcelos})}
\label{thm.CorsoHunekeKatzVasconcelos}
Suppose that $(R, \fram)$ is a Noetherian local ring and that $I$ is integrally closed and $\fram$-primary. Then $M$ has
projective dimension less than $t$ if and only if $\Tor^R_t(R/I,M) = 0$.
\end{theorem}

We specialize it in the following corollary that we will use in the next section.
\begin{corollary}
\label{cor.CorsoHunekeKatzVasconcelos}
Suppose that $(R, \fram)$ is a Noetherian local ring and that $I$ is an integrally closed $\fram$-primary ideal of finite projective dimension.  Then $R$ is regular.
\end{corollary}
\begin{proof}
Since $I$ has finite projective dimension, we see that $\Tor^R_i(R/I,k) = 0$ for $i \gg 0$.  But now taking $M = k = R/\fram$ in the statement of \autoref{thm.CorsoHunekeKatzVasconcelos} we see that $k$ has finite projective dimension since $I$ is integrally closed and $\fram$-primary.  The result follows.
\end{proof}

\subsection{Multiplier ideals and multiplier submodules}

For references in this section, see \cite{LazarsfeldPositivity2,SchwedeTakagiRationalPairs,BlickleSchwedeTuckerTestAlterations}.

\begin{definition}[Multiplier submodules]
Suppose that $\pi : X \to \Spec R$ is a resolution of singularities.  Then the \emph{multiplier submodule} of $R$, denoted $\mJ(\omega_R)$ is just $\pi_* \omega_X \subseteq \omega_R$.  Here $\omega_R$ (respectively $\omega_X$) is the first nonzero cohomology of the dualizing complex.

We now generalize this a bit.  Suppose $R$ is a normal domain, $\Gamma \geq 0$ is a $\bQ$-Cartier divisor, and $\pi$ is a log resolution of $(X, \Gamma)$.  Then we define $\mJ(\omega_R, \Gamma) = \pi_* \O_X(\lceil K_X - \pi^* \Gamma \rceil)$.  If we choose $0 \neq f \in R$ and $t \in \bQ_{\geq 0}$, then we set $\mJ(\omega_R, \Gamma, f^t) = \mJ(\omega_R, \Gamma + t \Div(f))$.  Finally, if $\fra \subseteq R$ is an ideal and $\pi$ is a log resolution of $(R, \Gamma, \fra^t)$ with $\fra \cdot \O_X = \O_X(-G)$, then we define
\[
\mJ(\omega_R, \Gamma, \fra^t) = \pi_* \O_X(\lceil K_X - \pi^* \Gamma - t G \rceil) \subseteq \omega_R.
\]
All of this is independent of the choice of resolution.
\end{definition}

In the above, if $\Gamma$ is ever left out, it is treated as zero.

\begin{definition}[Multiplier ideals]
Suppose that $R$ is a normal domain, $\Delta \geq 0$ is a $\bQ$-divisor such that $K_R + \Delta$ is $\bQ$-Cartier, $\fra \subseteq R$ is an ideal and $t \in \bQ_{\geq 0}$, then we define the \emph{multiplier ideal}
\[
\mJ(R, \Delta, \fra^t) = \pi_* \O_X(\lceil K_X - \pi^*(K_R + \Delta) - tG\rceil)
\]
where $\pi : X \to \Spec R$ is a log resolution of $(R, \Delta, \fra^t)$ and $\fra \cdot \O_X  = \O_X(-G)$.
Again, this is independent of the choice of resolution.
\end{definition}

If $\Delta$ is left off, then it is treated as zero and if $\fra$ is left off, it is treated as $R$.

\section{The main result in characteristic zero}

We begin with the ``easy'' direction.

\begin{theorem}
Suppose that $(R, \fram)$ is a regular local ring essentially of finite type over a field of characteristic zero.  If $\pi : X \to \Spec R$ is a regular alteration, then $\pd_R \myR \pi_* \O_X = 0$.
\end{theorem}
\begin{proof}
Since $R$ is regular, the bounded complex  $\myR \pi_* \O_X$ has finite projective dimension. By \autoref{thm.BuchEisComplex}, taking $M = R$ and $P = \myR \pi_* \O_X$ we have that
\[
\pd_R \myR \pi_* \O_X + \depth_R(\myR \pi_* \O_X) = \dim R.
\]
By the Matlis-dual version of Grauert-Riemenschneider vanishing \cite{GRVanishing}, $H^i_{\fram}(\myR \pi_* \O_X) = 0$ for all $i < \dim R$ and hence $\depth_R(\myR \pi_* \O_X) \geq \dim R$. Thus $\pd_R \myR \pi_* \O_X\leq 0$. But clearly $\pd_R \myR \pi_* \O_X\geq 0$ since $\myH^0(\myR \pi_* \O_X)\neq 0$. The result follows.
\end{proof}

\begin{lemma}
\label{lem.FPDImpliesCM}
Suppose $(R, \fram)$ is a local domain essentially of finite type over a field of characteristic zero and that $\pi : X \to \Spec R$ is a resolution of singularities.  If $\pd_R \myR \pi_* \O_X < \infty$, then $R$ is Cohen-Macaulay.
\end{lemma}
\begin{proof}
Since $\pd_R \myR \pi_* \O_X < \infty$, we see that the injective dimension of the Grothendieck dual, $\myR \pi_* \omega_X^{\mydot} \cong \pi_* \omega_X[\dim R]$ (by Grauert-Riemenschneider vanishing), is finite.  But then $\pi_* \omega_X$ is a finitely generated $R$-module of finite injective dimension and so $R$ is Cohen-Macaulay by Bass' question \cite{PeskineSzpiroDimensionProjective,HochsterTopicsInTheHomologicalTheory,RobertsIntersectionTheorems}.
\end{proof}
\begin{proof}[Alternate proof]
By the Matlis-dual version of the Grauert-Riemenschneider vanishing, we see that $H^i_{\fram}(\myR \pi_* \O_X) = 0$ for all $i < \dim R$.  Hence $\depth \myR \pi_* \O_X = \dim R$. Note also that $\pd_R \myR \pi_* \O_X \geq 0$ since $\myH^0(\myR \pi_* \O_X)\neq 0$. Thus we have
\[
\depth R = \depth(\myR \pi_* \O_X) + \pd_R(\myR \pi_* \O_X) \geq \dim R
\]
by \autoref{thm.BuchEisComplex} and hence $R$ is Cohen-Macaulay.
\end{proof}

We are ready to prove the following characterization of rational singularities, this result is an important step towards proving the main theorem and is interesting in its own right.

\begin{theorem}
\label{thm.RationalSing}
Suppose $(R, \fram)$ is a local domain essentially of finite type over a field of characteristic zero.  Let $\pi : X \to \Spec R$ be a resolution of singularities.  Then $R$ has rational singularities if and only if $\pd_R \myR \pi_* \O_X<\infty$.
\end{theorem}
\begin{proof}
If $R$ has rational singularities then obviously $\pd_R \myR \pi_* \O_X<\infty$ since $\myR \pi_* \O_X\cong R$. We now assume that $\pd_R \myR \pi_* \O_X<\infty$. We already see that $R$ is Cohen-Macaulay by \autoref{lem.FPDImpliesCM}. Hence, it is sufficient to show that $\pi_* \omega_X = \omega_R$.

So we suppose $\pi_* \omega_X \neq \omega_R$. By choosing a minimal prime $P$ of $\Supp (\omega_R/\pi_*\omega_X)$ and replacing $R$ by $R_P$, we may assume $R$ has rational singularities on the punctured spectrum (i.e., $\omega_R/\pi_*\omega_X$ has finite length). Since $\pi_* \omega_X$ has finite injective dimension (see the proof of \autoref{lem.FPDImpliesCM}), by \cite[Theorem 2.9]{SharpFinitelyGeneratedmodulesFiniteInjectiveDimension}, $\Hom_R(\omega_R, \pi_* \omega_X)$ has finite projective dimension.  But
\[
\Hom_R(\omega_R, \pi_* \omega_X) = \pi_* \sHom_{\O_X}(\pi^* \omega_R, \omega_X).
\]
Now $\sHom_{\O_X}(\pi^* \omega_R, \omega_X)$ is a rank $1$ reflexive sheaf on $X$. Since $X$ is regular, $\sHom_{\O_X}(\pi^* \omega_R, \omega_X)$ is locally free and so its pushforward, which is isomorphic to $$\Hom_R(\omega_R, \pi_* \omega_X) \subseteq \Hom_R(\omega_R, \omega_R) \subseteq R,$$ is an integrally closed ideal. Since our assumption is $0\neq \omega_R/\pi_*\omega_X$ has finite length, it follows that $\Hom_R(\omega_R, \pi_* \omega_X)\neq R$ is an $\m$-primary integrally closed ideal. But then by \autoref{cor.CorsoHunekeKatzVasconcelos}, $\pd_R\Hom_R(\omega_R, \pi_* \omega_X)<\infty$ already implies $R$ is regular and thus $\pi_* \omega_X =\omega_R$ which is a contradiction.
\end{proof}
\begin{remark}
Bhargav Bhatt communicated to us an alternate proof of \autoref{thm.RationalSing}, which we now sketch.
Since $\myR \pi_* \O_X$ is a perfect complex, there exists a trace map
\[
\myR \Hom_R(\myR \pi_* \O_X, \myR\pi_* \O_X) \to R.
\]
On the other hand, we have the map $\myR \pi_* \O_X \to \Hom_R(\myR \pi_* \O_X, \myR\pi_* \O_X)$ coming from $\O_X$'s left multiplication action on itself.  We have the composition
\[
R \to \myR \pi_* \O_X \to \myR \Hom_R(\myR \pi_* \O_X, \myR\pi_* \O_X) \to R
\]
which is an isomorphism generically (on the open set where $\pi$ is an isomorphism), hence an isomorphism.  But then $R$ has rational singularities by \cite{KovacsRat} (note that that result still utilizes Grauert-Riemenschneider vanishing).
\end{remark}

\subsection{An aside on multiplier ideals}

We assume the following is essentially well known to experts, but we do not know a reference.

\begin{proposition}
\label{prop.MultiplierIdeal}
Suppose $(R,\fram)$ is a normal local domain essentially of finite type over a field of characteristic zero. Suppose $0\neq f\in R$ such that $\Div_R(f)$ is reduced. Fix $N\geq 0$ and let $S=R[f^{1\over N+1}]$ be the normal cyclic cover. Then $\mJ(\omega_S)$ has an $R$-summand isomorphic to $\mJ(\omega_R, f^{N \over N+1})$.
\end{proposition}

\begin{proof}
Since $\Div_R(f)$ is reduced, $S$ is regular in codimension $1$ and hence $S$ is normal. Choose $-K_R$ effective.
By \cite[Theorem 9.5.42]{LazarsfeldPositivity2} (see also \cite[Theorem 8.1]{BlickleSchwedeTuckerTestAlterations}) we see that
\[
\mJ(R, -K_R + {N \over N+1} \Div_R f) = R \cap \mJ(S, -\Ram_{S/R}  - \rho^* K_R + {N \over N+1} \Div_S f ).
\]
Again since $\Div_R f$ is reduced, we see that $\Ram_{S/R} = {N \over N+1} \Div_S f$ and hence
\[
\mJ(R, -K_R + {N \over N+1} \Div_R f) \subseteq \mJ(S, -\rho^* K_R).
\]
On the other hand, by \cite[Theorem 8.1]{BlickleSchwedeTuckerTestAlterations}, we have a splitting (up to scalars)
\[
\begin{array}{rl}
& \Tr(\mJ(S, -\rho^* K_R)) \\
= & \Tr(\mJ(S, -\rho^* K_R - \Ram_{S/R} + {N \over N+1} \Div_S f )) \\
= & \mJ(R, -K_R + {N \over N+1} \Div_R f)\\
 = & \mJ(\omega_R, f^{N \over N+1}).
\end{array}
\]
But we have
\[
\begin{array}{rl}
& \mJ(S, -\rho^* K_R)\\
 = & \mJ(S, -\rho^* K_R - \Ram_{S/R} + {N \over N+1} \Div_S f ) \\
 = & \mJ(S, -K_S + {N \over N+1} \Div_S f ) \\
 = & \mJ(\omega_S, f^{N \over N+1}).
 \end{array}
\]
We have just shown that $\mJ(\omega_S, f^{N \over N+1})$ has an $R$-summand isomorphic to $\mJ(\omega_R, f^{N \over N+1})$.  But even as an $S$-module $\mJ(\omega_S, f^{N \over N+1}) = f^{N \over N+1} \mJ(\omega_S) \cong \mJ(\omega_S)$, and hence $\mJ(\omega_S)$ has an $R$-summand isomorphic to $\mJ(\omega_R, f^{N \over N+1})$.
\end{proof}

\subsection{Proof of Main Theorem in characteristic zero}

We now complete the proof of our main result in characteristic zero.

\begin{theorem}
Suppose $(R, \fram)$ is a local domain essentially of finite type over a field of characteristic zero.  Suppose that for every regular alteration $\pi : X \to \Spec R$, $\pd_R \myR \pi_* \O_X<\infty$.  Then $R$ is regular.
\end{theorem}
\begin{proof}
By \autoref{thm.RationalSing}, we already know that $R$ has rational singularities.  Choose $N >0$ so that $\mJ(\omega_R, \fram^N)\neq \omega_R$.  Then choose a general $f \in \fram^{N+1}$ and by \cite[Proposotion 9.2.28]{LazarsfeldPositivity2} we know that $\mJ(R, -K_R + {N \over N+1} \Div_R f ) = \mJ(R, -K_R, \fram^N) = \mJ(\omega_R, \fram^N)$.

Consider the normal cyclic cover $S = R[f^{1\over N+1}]$. Since $f$ is general, $\Div_R(f)$ is reduced and by \autoref{prop.MultiplierIdeal}, we know that $\mJ(\omega_S)$ has an $R$-summand isomorphic to $\mJ(\omega_R, f^{N \over N+1})$.

Next consider a resolution of singularities $\pi$: $X\to \Spec S$, then the composition $X\to \Spec S\to \Spec R$ is a regular alteration. Moreover, $\pi_*\omega_X=\mJ(\omega_S)$ has finite injective dimension over $R$ (because $\pi_*\omega_X[\dim R]$ is the Grothendieck dual of $\myR\pi_*\O_X$), so does its direct summand $\mJ(\omega_R, f^{N \over N+1})$. Therefore by \cite[Theorem 2.9]{SharpFinitelyGeneratedmodulesFiniteInjectiveDimension}, $$\Hom_R(\omega_R, \mJ(\omega_R, f^{N \over N+1})) \subseteq \Hom_R(\omega_R, \omega_R) \cong R$$ has finite projective dimension.  Since $\mJ(\omega_R, f^{N \over N+1})=\mJ(\omega_R,\m^N)$ agrees with $\omega_R$ except at the origin (where it \emph{does not} agree).  Thus
$\Hom_R(\omega_R, \mJ(\omega_R, f^{N \over N+1}))$ lacks the identity map $\omega_R \to \omega_R$ and hence it is identified with an $\fram$-primary ideal of $R$.

Next we show that $\Hom_R(\omega_R, \mJ(\omega_R, f^{N \over N+1})) \subseteq R$ is an integrally closed ideal.  Take a log resolution of singularities $\pi : X \to \Spec R$ of $(R, \Div_R(f))$. By definition we have $\mJ(\omega_R, f^{N \over N+1}) = \pi_* \O_X(\lceil K_X - {N \over N+1} \Div_X(f)\rceil )$.  Thus
\[
\begin{array}{rl}
& \Hom_R(\omega_R, \mJ(\omega_R, f^{N \over N+1}))  \\
= &\Hom_R(\omega_R, \pi_* \O_X(\lceil K_X - {N \over N+1} \Div_X(f)\rceil )) \\
= & \pi_* \sHom_{\O_X}(\pi^* \omega_R, \O_X(\lceil K_X - {N \over N+1} \Div_X(f)\rceil )).
\end{array}\]
As in the proof of \autoref{thm.RationalSing}, since $\sL := \sHom_{\O_X}(\pi^* \omega_R, \O_X(\lceil K_X - {N \over N+1} \Div_X(f)\rceil ))$ is a rank 1 reflexive sheaf and $X$ is regular, $\sL$ is invertible.  Thus $\Hom_R(\omega_R, \pi_* \O_X(\lceil K_X - {N \over N+1} \Div_X(f)\rceil ))$ is an integrally closed $\m$-primary ideal of finite projective dimension.  Therefore $R$ is regular by
\autoref{cor.CorsoHunekeKatzVasconcelos}.
\end{proof}

\begin{remark}
\label{rmk.MixedCharacteristicSurface}
We believe that the above proof can be run (essentially without change) for excellent surfaces even in mixed characteristic.
The key facts we need are that Grauert-Riemenschneider still holds for excellent surfaces \cite[Corollary 2.10]{LipmanDesingularizationOf2Dimensional} and that we can choose a general element $f$ in $\fram^{N+1}$ so that $\mJ(R, f^{N/N+1})$ is $\fram$-primary \cite{TrivediLocalBertini,TrivediLocalBertiniErratum} (using that $R$ is regular outside of the origin since we may reduce to the case that $R$ is normal).
\end{remark}
\bibliographystyle{alpha}
\bibliography{MainBib}
\end{document}